\documentclass[reqno]{amsart}
 \usepackage{cleveref}
 \crefformat{section}{\S#2#1#3} 
\usepackage{mathtools}
\usepackage{amssymb, amsmath}
\usepackage{leftidx}
\usepackage{forest}
\usepackage{color}
\newtheorem{theorem}{Theorem}[section]
\newtheorem{lemma}[theorem]{Lemma}
\newtheorem{prop}[theorem]{Proposition}
\theoremstyle{definition}
\newtheorem{definition}[theorem]{Definition}
\newtheorem{example}[theorem]{Example}

\newtheorem{cor}[theorem]{Corollary}

\theoremstyle{remark}
\newtheorem{remark}[theorem]{Remark}

\numberwithin{equation}{section}

\newcommand{\mf}{\mathfrak}
\newcommand{\mc}{\mathcal}

\newcommand{\C}{\mathbb{C}}

\newcommand{\Z}{\mathbb{Z}}

\newcommand{\Span}{{\rm span}}

\newcommand{\Aut}{{\rm Aut}}
\newcommand{\Cur}{{\rm Cur}}

\newcommand{\Vir}{{\rm Vir}}

\def \bp{\begin{prop}\label}
\def \ep{\end{prop}}
\def \bt{\begin{theorem}\label}
\def \et{\end{theorem}}

\begin{document}

\title{Classification of rank two Lie conformal algebras}

\author{Rekha Biswal}
\address{Max Planck Institute for Mathematics, Bonn, Germany}
\email{rekha@mpim-bonn.mpg.de}

\author{Abdelkarim Chakhar}
\address{ D\'epartement de math\'ematiques et de statistique, Universit\'e Laval,  Qu\'ebec, QC, Canada}
\email{abdelkarim.chakhar.1@ulaval.ca}

\author{Xiao He}
\address{School of Mathematics (Zhuhai), Sun Yat-Sen University, Zhuhai, China.}
\email{hexiao8@mail.sysu.edu.cn(corresponding email)}




\begin{abstract}
We give a complete classification (up to isomorphism) of rank two Lie conformal algebras,  and determine their automorphism groups. 
\end{abstract}

\maketitle

\section*{Introduction}

Lie conformal algebras, and their super versions, were first introduced rigorously by V. Kac \cite{Kac96} around 1996 .  Their structure theory, representation theory and cohomology theory have been studied extensively by V. Kac and his collaborators in \cite{BakKacVoronov, ADA&Kac98, Fattori&Kac, FattoriKacRetakh, Kac96}. A non-linear version of Lie conformal algebras was further studied in \cite{BakDeSole, DeSoleKac}. On one hand, Lie conformal algebras axiomatize the singular part of the operator product expansion (OPE) of chiral fields in a two-dimensional conformal field theory. 
On the other hand, they have been proved to be useful in the study of infinite-dimensional Lie algebras satisfying the locality property \cite{KacLocality}.  Lie conformal algebras are  related to vertex algebras in a similar manner as Lie algebras are related to their universal enveloping algebras \cite{BakalovKacFieldalg, Kac96}.  

Besides the applications to other areas, Lie conformal algebras themselves are quite interesting from a purely algebraic point of view.  As algebraic objects, their definition has a lot of resemblance to that of Lie algebras. A Lie algebra  is a vector space endowed with a Lie bracket satisfying skew-symmetry and the Jacobi identity. In the case of Lie conformal algebras, several things need to be modified. First, the base ring is replaced by the polynomial ring $\C[\partial]$ in in one variable. The base space $R$ is replaced by  a module over $\C[\partial]$ and Lie bracket is replaced by lambda-bracket $[\cdot_\lambda \cdot]: R\otimes R \rightarrow \C[\lambda]\otimes R$, for some indeterminate $\lambda$. Then the corresponding axioms of sesquilinearity,  skew-symmetry and Jacobi identity  for the lambda-bracket versions are required to be satisfied (see  \Cref{def:LCA}).

A Lie conformal algebra $R$ is said to be finite if it is finitely generated as a $\C[\partial]$-module, and it is said to be of rank $n$ if it is free of rank $n$ as a $\C[\partial]$-module. As in the case of Lie algebras, the corresponding notions of simple, semisimple, nilpotent and solvable Lie conformal algebras can be defined.  A complete classification of finite semisimple Lie conformal algebras was done in \cite{ADA&Kac98}. A super version of the classification was done soon after in \cite{Fattori&Kac}.

Let $R$ be a $\C[\partial]$-module of rank $n$. Defining a Lie conformal algebra structure on $R$ is equivalent to finding a set of polynomials, which we call the structure polynomials analogous to the structure constants in the case of Lie algebras, satisfying certain  equations, i.e., the structure polynomial versions for skew-symmetry and Jacobi identity equations (see (\ref{eq:skewsymmetryranktwo}), (\ref{eq:Jacobigeneral}) for the rank two case, or (2.7.8) and (2.7.9) in \cite{Kac96} for the general case). In \cite{Kac96}, V. Kac and M. Wakimoto classified the rank one Lie conformal algebras by giving explicit solutions of those equations. Further, V. Kac surmised in \cite{Kac96} that ``It is clearly impossible to solve those  equations directly for $n\geq 2$". In this paper, we give a solution to the rank two case.

Our main result, \Cref{maintheorem}, may be stated briefly as follows.   Let $R$ be a non-semisimple rank two  Lie conformal algebra:

\begin{enumerate}
	\item If $R$ is  nilpotent, then $R$ is parametrized by a skew-symmetric polynomial $Q(\lambda, \partial)$, i.e. $Q(\lambda, \partial)=-Q(-\lambda-\partial, \partial)$. If $R$ is solvable but not nilpotent, then $R$ is parametrized by a non-zero polynomial $a(\lambda)$ (\Cref{nilpotent}).
	
 	\item If $R$ is not solvable, then either $R$ is the direct sum of a rank one commutative Lie conformal algebra and the Virasoro Lie conformal algebra (\Cref{semidirect}), or it belongs to a class of Lie conformal algebras parametrized by three parameters  $c, d,\text{ and } Q_c(\lambda, \partial))$, where $c, d\in \C$ are constants and $Q_c(\lambda, \partial)$ is a skew-symmetric polynomial. Moreover,  $Q_c(\lambda, 
 	\partial)\neq 0$ only when $(c, d)\in \{(1, 0), (0, 0), (-1, 0), (-4, 0), (-6, 0)\}$. In these cases, we document the polynomials $Q_c(\lambda, \partial)$ explicitly in the table of \Cref{maintheorem}.
\end{enumerate}

The organization of the present paper is as follows. In \cref{sec:1}, we recall the definition of a Lie conformal algebra and give some concrete examples. In \cref{sec:2}, we give a complete classification of rank two Lie conformal algebras.  We compute the automorphism groups  of rank two Lie conformal algebras  in \cref{sec:3}. All vector spaces and tensor products are considered over the complex numbers $\C$. 
\medskip

\emph{Acknowledgements.} X. He would like to thank the China Scholarship Council (File No.201304910374) and l'Institut des sciences math\'ematiques for their financial support. The three authors also gratefully acknowledge funding received from the NSERC Discovery Grant of their research supervisor (Michael Lau). The authors would like to thank Prof. V. Kac for valuable discussion after the completion of the present paper and also for bringing \cite{Kacnotes, Ritt} to their attention. 

\bigskip

\section{Preliminaries}\label{sec:1}

\begin{definition}[\cite{Kac96})]\label{def:LCA}
A {\itshape Lie conformal algebra} is 
a  $\C[\partial]$-module $R$ 
endowed with a $\C$-linear map $R\otimes R \rightarrow \C[\lambda]\otimes R$, 
denoted by $a\otimes b \mapsto [a _\lambda  b]$ 
and called the $\lambda$-bracket, satisfying the following axioms:  for all $a,b,c\in R$, we have
\begin{align}
&[\partial a _\lambda b] = -\lambda[a  _\lambda   b] 
  ,\quad [a _\lambda \partial b] = (\lambda+\partial)[a _\lambda  b], \qquad \mbox{\tag*{(sesquilinearity)}} \\
&[b _\lambda   a]= -[a  _{-\lambda-\partial} b] ,  \qquad \qquad \qquad \qquad \qquad \qquad \mbox{\tag*{(skew-symmetry)}} \\
&[a _\lambda [b _\mu  c]]-[b _\mu [a _\lambda c]] =
   [[a _\lambda  b]  _{\lambda+\mu}  c]. \qquad \qquad \qquad  \mbox{\tag*{(Jacobi identity)}}
\end{align}

\end{definition}
\bigskip

The sesquilinearity implies that,  $\mbox{~for all~} a, b\in R \mbox{~and~} f(\partial), g(\partial)\in \C[\partial]$, we have
\begin{equation}\label{eq:sesquilinearity}
[f(\partial)a_\lambda g(\partial)b]=f(-\lambda)g(\lambda+\partial)[a_\lambda b].
\end{equation}

Given a Lie conformal algebra $R$, for each $n\in \{0,1,2,\cdots\}$, the \emph{$n$-th product} is defined as a $\C$-linear product 
$R\otimes R \rightarrow  R$  and denoted by
$a\otimes b \mapsto a_{(n)} b$, which is encoded in the $\lambda$-bracket as:
\begin{equation}\label{nthprod}
[a  _\lambda  b]=\sum_{n\in\Z_+}\lambda^{(n)}a_{(n)}b \ ,
\end{equation}
where we use the notation $\lambda^{(n)}=\frac{\lambda^n}{n!}$ and $\lambda^{(n)}a_{(n)}b:=\lambda^{(n)}\otimes a_{(n)}b$. Then we have the following equivalent definition.

\begin{definition}[\cite{Kac96}]\label{def:LCA2}
A Lie conformal algebra  is a $\C[\partial]$-module $R$ endowed with a $\C$-linear product $-_{(n)}-$ for each $n\in \Z_+$ satisfying the following axioms: for all $a, b, c\in R$ and $ m, n\in \Z_+$,
\begin{align*}
&(C0)\quad  a_{(n)}b=0 \mbox{~for~} n\gg 0, \\
&(C1) \quad \partial a_{(n)} b= -na _{(n-1)}b 
  ,\quad a _{(n)}\partial b=\partial(a_{(n)}b)+na _{(n-1)}b,\\
&(C2)\quad  b_{(n)}a= -\sum_{j\geq 0}(-1)^{n+j}\partial^{(j)}(a_{(n+j)}b), \quad\text{~here~} \partial^{(j)}=\dfrac{\partial^j}{j!}, \\
&(C3)\quad a _{(m)}(b _{(n)} c)-b _{(n)} (a _{(m)}c )= \sum_{j\geq 0}^m \binom{m}{j}(a_{(j)}b)_{(m+n-j)}c.
\end{align*}

\end{definition}

\begin{definition}
Let $R_1, R_2$ be two Lie conformal algebras. A Lie conformal algebra homomorphism $\varphi: R_1\rightarrow R_2$ is a $\C[\partial]$-module homomorphism which preserves the $\lambda$-brackets, i.e., $\varphi([a_\lambda b])=[\varphi(a)_\lambda \varphi(b)]$ for all $a, b\in R_1$, where it is understood that $\varphi$ commutes with multiplication by $\lambda$. 
\end{definition}

\begin{definition}
Let $R$ be a Lie conformal algebra, and $S\subseteq R$ a $\C[\partial]$-submodule. Then $S$ is called a {\em Lie conformal subalgebra} of $R$ if $a_{(n)}b\in S$ for all $a, b\in S, n\in \Z_+$. It is called a {\em Lie conformal ideal} of $R$ if $a_{(n)}b\in S$ for all $a\in R, b\in S, n\in \Z_+$, and we denote it by $S\vartriangleleft R$. By the  skew-symmetry ($C1$), an ideal is always two-sided.  
\end{definition}

\begin{lemma}\label{idealbracket}
	Let $R$ be a Lie conformal algebra and $S, T$ be two ideals of $R$. The subspace $[S_\lambda T]:=\Span_\C\{a_{(n)}b~|~ a\in S, b\in T, n\in \Z_{\geq 0}\}$ is an ideal of $R$. 
\end{lemma}
\begin{proof}
	This can be easily seen by the Jacobi identity $(C3)$ in \Cref{def:LCA2}. 
\end{proof}

If $I\subseteq R$ is an ideal, then there is a canonical way to define a Lie conformal algebra structure on the quotient $R/I$.  The {\em derived algebra} of $R$, which we denote by $R'$, is defined to be the $\C$-span of all elements of the form $a_{(n)}b$ with $a, b\in R, n\in \Z_+.$   We denote by $R^{(1)}=R$  and define $R^{(n)}=(R^{(n-1)})'$ inductively for $n\geq 2$, then we have the derived algebra series, which are all ideals by \Cref{idealbracket},
$$R\supseteq R^{(2)}\supseteq \cdots \supseteq R^{(n)}\supseteq R^{(n+1)}\supseteq \cdots.$$
Let $R^n=\Span_\C \{a^1_{(m_1)}a^2_{(m_2)}\cdots a^{n-1}_{(m_{n-1})}a^{n}~|~a^i\in R, m_i\in \Z_+\}$, then we have $R^{(n)}\subseteq R^n$ and the following series of ideals 
$$R\supseteq R^2\supseteq \cdots \supseteq R^{n}\supseteq R^{n+1}\supseteq \cdots.$$

\begin{definition} 
Let $R$ be a Lie conformal algebra. An element $a\in R$ is called {\em central} if $[a_\lambda b]=0$ for all $b\in R$, and $R$ is called \emph{commutative} if all elements of $R$ are central. If $R^{(n)}=0$ for some $n\geq 1$, $R$ is called \emph{solvable} and if $R^n=0$ for some $n\geq 1$, $R$ is called \emph{nilpotent}.  A Lie conformal algebra is called \emph{simple} if it is not commutative and has no proper ideals. It is called \emph{semisimple} if it contains no non-zero solvable ideals. 
\end{definition}

\begin{example}\label{Virasoro}
The {\em Virasoro Lie conformal algebra } is the rank one $\C[\partial]$-module $\Vir=\C[\partial]L$ with $[L_\lambda L]=(2\lambda+\partial)L.$ It is the unique non-commutative Lie conformal algebra of rank one. 
\end{example}

\begin{example}\label{Currents}
Given a finite dimensional Lie algebra $\mf g$,  $\Cur\, \mf g:=\C[\partial]\otimes \mf g$ can be given a Lie conformal algebra structure by defining $[a_\lambda b]=[a, b]$ for $a, b\in \mf g$, and it is called the {\em current Lie conformal algebra} associated to $\mf g$.  
\end{example}

\begin{example}\label{semidirectcurrentVirasoro}
The semi-direct sum of $\Vir$ with $\Cur\, \mf g$ is the Lie conformal algebra $R=\Vir\ltimes\Cur\, \mf g$ with
$[L_\lambda L]=(2\lambda+\partial)L, [L_\lambda c]=(\lambda+\partial)c$, and $[a_\lambda b]=[a, b]$ for all $a, b, c\in \mf g$.
\end{example}

In fact, the above three examples are sufficient to describe the classification of finite semisimple Lie conformal algebras. 
\begin{theorem}[\cite{ADA&Kac98}]\label{semisimpleLCA}
Any finite semisimple Lie conformal algebra is a finite direct sum of Lie conformal algebras of the following types:
\begin{itemize}
\item [(i)]  the finite current Lie conformal algebra $\Cur \,\mf g$  with $\mf g$ being simple,
\item[(ii)] the Virasoro Lie conformal algebra $\Vir$, 
\item[(iii)] the semidirect sum of $(i)$ and $(ii)$ as in \Cref{semidirectcurrentVirasoro}.
\end{itemize}
\end{theorem}

The goal of this paper is to classify rank two Lie conformal algebras.  Let $R=\C[\partial]X^1\oplus \C[\partial] X^2$.  To define a Lie conformal algebra structure on $R$, it is enough to define the $\lambda$-brackets of the basis elements $\{X^1, X^2\}$, namely, 
\begin{align}\label{eq:skewsymmetrygeneral}
[X^i  _\lambda  X^j]&=Q_{i, j}^k(\lambda, \partial)X^k, \mbox{~where~} Q_{i, j}^k(\lambda, \partial)\in \C[\lambda, \partial] \mbox{~and~}  i, j, k \in \{1, 2\},
\end{align}    
such that the skew-symmetry property for the pairs $(X^i, X^j)$  and the Jacobi identities for the triples $(X^i, X^j, X^k)$  are all satisfied. By sesquilinearity,  we can then extend the $\lambda$-brackets to $R$. We will call the polynomials $Q_{i, j}^k(\lambda, \partial)$ the \emph{structure polynomials} of $R$. They are similar to the structure constants of a Lie algebra under a chosen basis and completely determine the Lie conformal algebra structure on $R$.

Let us recall here the classification of rank one Lie conformal algebras following \cite{Kac96}. Let $R=\C[\partial]L$ be a rank one Lie conformal algebra, with $[L_\lambda L]=f(\lambda, \partial)L$ for some $f(\lambda, \partial)$. Then by skew-symmetry and Jacobi identity, we have \begin{equation*}\label{skewsymmetryrankone}
f(\lambda, \partial)=-f(-\lambda-\partial,\partial),
\end{equation*}  
\begin{equation}\label{Jacobirankone}
f(\mu, \lambda+\partial)f(\lambda, \partial)-f(\lambda, \mu+\partial)f(\mu, \partial)=f(\lambda, -\lambda-\mu)f(\lambda+\mu, \partial).
\end{equation}
Let $\deg_\partial f(\lambda, \partial)=n$. If $n\geq 2$, an easy calculation shows that the degree with respect to $\partial$ on the left hand side of (\ref{Jacobirankone}) is $2n-1$, while that on the right hand side is $n$, which is impossible since $2n-1\neq n$. So $n\leq 1$.  If $n=0$,  skew-symmetry implies that $f=0$, so we have the commutative Lie conformal algebra. If $n=1$, then  $f(\lambda, \partial)=(2\lambda+\partial)\alpha$  for some non-zero  constant $\alpha\in \C^\times$ by  skew-symmetry. We can then do change of basis to set $\alpha=1$, i.e., we have the Virasoro Lie conformal algebra. So a rank one Lie conformal algebra is either commutative or the Virasoro Lie conformal algebra.

\begin{definition}
A polynomial $f(x, y)\in \C[x, y]$ is said to be {\em skew-symmetric} if it satisfies the equation $$f(x, y)=-f(-x-y, y).$$  
\end{definition}
 Note that the polynomials $Q_{1, 1}^1(\lambda, \partial), Q_{1, 1}^2(\lambda, \partial), Q_{2, 2}^1(\lambda, \partial)$ and $Q_{2, 2}^2(\lambda, \partial)$ defined in (\ref{eq:skewsymmetrygeneral}) are all skew-symmetric. 
\begin{lemma}\label{Skew-symmetry}
If $f(x, y)\in \C[x, y]$ is a skew-symmetric polynomial,  then $f(x, y)=(2x+y)g(x^2+xy, y)$ for some polynomial $g$. 
\end{lemma}
\begin{proof}
	It is easy to see that $(2x+y)|f(x, y)$. Indeed, plugging $y=-2x$ in $f(x, y)$,  we have $f(x, -2x)=-f(x, -2x)$ by skew-symmetry, i.e., $f(x, -2x)=0$.  Let $f(x, y)=(2x+y)g(x, y)$, then $g(x, y)=g(-x-y, y)$. Hence $g(x, y)$ is invariant under the transformation $x\mapsto -x-y, y\mapsto y$. Thus an easy application of the invariant theory implies that $g(x, y)\in \C[x^2+xy, y]$. 
\end{proof}

\begin{cor}\label{keycorollary}
Let $f(x, y)$ be a skew-symmetric polynomial. If $\deg_y f(x, y)=n$, then the coefficient of $y^n$ is a polynomial in $x$ of degree $\leq n-1$.  
\end{cor}
\begin{proof}
By skew-symmetry, we can write $f(x, y)=(2x+y)g(x^2+xy, y)$ for some polynomial $g$, and $\deg_y g(x^2+xy, y)=n-1$.  As $\deg_y (x^2+xy)=1$, we can write $g(x^2+xy, y)=\sum_{i+j\leq n-1} b_{i,j}(x^2+xy)^i y^j$.  Thus the coefficient of $y^n$ in $f(x, y)$ is $\sum_{i=0}^{n-1}b_{i, n-1-i} x^i$, which is a polynomial in $x$ of degree less than or equal to $n-1$.
\end{proof}
\bigskip
\section{Classification of rank two Lie conformal algebras}\label{sec:2}

\subsection{\textbf{The first step towards the classification}}

For a rank two Lie conformal algebra $R$ with basis $\{X^1, X^2\}$ and structure polynomials $Q_{i, j}^k(\lambda, \partial)$, by the skew-symmetry and Jacobi identity axioms,  the structure polynomials must satisfy
\begin{equation}\label{eq:skewsymmetryranktwo}
Q_{i, j}^k(\lambda, \partial)=-Q_{j, i}^k(-\lambda-\partial, \partial) \mbox{~for~} i, j, k\in \{1, 2\},\end{equation}
\begin{align}\label{eq:Jacobigeneral}
\sum_{s=1}^2\left(Q_{j,k}^s(\mu, \lambda+\partial)Q_{i, s}^t(\lambda, \partial)-Q_{i,k}^s(\lambda, \mu+\partial)Q_{j, s}^t(\mu, \partial)\right)\nonumber\\
=\sum_{s=1}^2Q_{i, j}^s(\lambda, -\lambda-\mu)Q_{s, k}^t(\lambda+\mu, \partial)
\end{align}
for all $i, j, k, t\in \{1, 2\}$. Although the Jacobi identity is $\mc S_3$-symmetric, and we just need to consider (\ref{eq:Jacobigeneral}) for the triples $(i,j,k)$ with $i\leq j\leq k$, there are still 6 equations containing 12 polynomials in two variables of the same type as (\ref{eq:Jacobigeneral}).

So it is important to simplify the calculations if we want to do the classification. From \Cref{semisimpleLCA}, we know that the only semisimple Lie conformal algebra of rank two is the direct sum of two Virasoro Lie conformal algebras. So we only need to  consider the non-semisimple ones. The following key lemma plays a crucial role in simplifying the calculations in our classification. 
\begin{lemma}\label{keylemma}
For a non-semisimple Lie conformal algebra $R$ of rank two, there exists a basis, say $\{A, B\}$, such that $[A_\lambda A]=0$, and  $\C[\partial]A\vartriangleleft R$ is an ideal.
\end{lemma}
\begin{proof}
Since $R$ is not semisimple, there exists a non-zero  solvable ideal. Let $S\subseteq R$ be such an ideal. Then the derived series of $S$  must terminate somewhere, i.e.,  $S^{(n)}\neq 0$ but $S^{(n+1)}=0$ for some $n\in \Z_+$. Then $I:=S^{(n)}$ is commutative. By \Cref{idealbracket}, all these $S^{(i)}$'s are  ideals of $R$.  In particular, $I$ is an ideal of $R$. Since $\C[\partial]$ is a principal ideal domain and $I$ is a submodule of $R$,  we can find a basis of $R$, say $\{A, B\}$, such that $I$ is generated by $\{f(\partial)A, g(\partial)B\}$ for some $f(\partial), g(\partial)\in \C[\partial]$, where $f(\partial)$ is a non-zero polynomial and $f(\partial)$ divides $g(\partial)$. When $g(\partial)\neq 0$, i.e, $I$ is of rank two as a $\C[\partial]$-module,  using (\ref{eq:sesquilinearity}), it is straightforward to see that $R$ is commutative . When $g(\partial)=0$, we have $I=\C[\partial]f(\partial)A$ is a rank one abelian ideal. Thus $\C[\partial]A$ is an abelian ideal of $R$. 
\end{proof}

From now on, we assume that $R$ has a basis $\{A, B\}$ satisfying $[A_\lambda A]=0$ and  $\C[\partial]A \vartriangleleft R$.  Comparing the coefficients of $A$ in the Jacobi identity for the triple $(B, B, A)$, we have  
\begin{align}\label{JacobiBBA}
   Q_{B,A}^A(\mu,\lambda+\partial)Q_{B,A}^A&(\lambda,\partial)-Q_{B,A}^A(\lambda,\mu+\partial)Q_{B,A}^A(\mu,\partial)\nonumber\\
   &=Q_{B,B}^B(\lambda,-\lambda-\mu)Q_{B,A}^A(\lambda+\mu,\partial),
\end{align}
where we assume that $[X_\lambda Y]=Q_{X, Y}^A(\lambda, \partial)A+Q_{X, Y}^B(\lambda, \partial)B$ for $X, Y\in \{A, B\}$. 
Using the same argument as in the classification of rank one Lie conformal algebras, we can easily see that $\deg_\partial Q_{B, A}^A(\lambda, \partial)\leq 1$, so we may assume that $Q_{B, A}^A (\lambda, \partial)=a(\lambda)+b(\lambda)\partial$ for some polynomials $a(\lambda), b(\lambda)\in \C[\lambda]$.
Since $\C[\partial]A$ is an ideal, $\dfrac{R}{\C[\partial]A}\cong \C[\partial]B$ is a rank one Lie conformal algebra. Thus $[B_\lambda B]=\alpha(2\lambda+\partial)B \mod \C[\partial]A$ for some scalar $\alpha$. Without loss of generality (after doing some scaler multiplication), we can assume moreover that $\alpha\in \{0, 1\}$. 

\begin{lemma}\label{conditions}Let $R=\C[\partial]A\oplus \C[\partial]B$ be a rank two Lie conformal algebra with $[A_\lambda A]=0, [B_\lambda A]=(a(\lambda)+b(\lambda)\partial )A$ and $Q_{B, B}^B(\lambda, \partial)=(2\lambda+\partial)\alpha$,  for some polynomials $a(\lambda), b(\lambda)\in \C[\lambda]$ and $\alpha\in \{0, 1\}$. Then the following holds,
\begin{itemize}
	\item[(1)] if $\alpha=0$, then $b(\lambda)=0$,
	\item[(2)] if $\alpha=1$, then $b(\lambda) \in \{0,1\}$. Moreover, if $b(\lambda)=0$, then we have $a(\lambda)=0$; if $b(\lambda)=1$, then $\deg a(\lambda)\leq 1$.
\end{itemize}
\end{lemma}
\begin{proof}
Set $Q_{B, A}^A(\lambda, \partial)=a(\lambda)+b(\lambda)\partial$ and $Q_{B, B}^B(\lambda, \partial)=(2\lambda+\partial)\alpha$ in 
(\ref{JacobiBBA}) yields
\begin{align}\label{eq:2.5}
(a(\mu)+b(\mu)&(\lambda+\partial))(a(\lambda)+b(\lambda)\partial)-(a(\lambda)+b(\lambda)(\mu+\partial))(a(\mu)+b(\mu)\partial)\nonumber\\ 
&=(\lambda-\mu)\alpha(a(\lambda+\mu)+b(\lambda+\mu)\partial). 
\end{align}
Comparing the coefficients of powers of $\partial$ in (\ref{eq:2.5}), we get	
\begin{equation}\label{ablambda1}
b(\lambda)b(\mu) = b(\lambda+\mu)\alpha.
\end{equation}
\begin{equation}\label{ablambda2}
b(\mu)a(\lambda)\lambda-b(\lambda)a(\mu)\mu=(\lambda-\mu)a(\lambda+\mu)\alpha
\end{equation}
When $\alpha=0$, we get $b(\lambda)=0.$ When $\alpha=1$, (\ref{ablambda1}) ensures that $b(\lambda)=k$ is a constant and $k^2=k$, i.e.,  $k\in \{0, 1\}$. When $b(\lambda)=0$, (\ref{ablambda2})  implies that $a(\lambda)=0$. When $b(\lambda)=1$, (\ref{ablambda2}) implies that $\deg a(\lambda)\leq 1$.
\end{proof}

Therefore,  any non-semisimple rank two Lie conformal algebra $R$ has a basis $\{A, B\}$, such that $[A_\lambda A]=0$,  $[B_\lambda A]=(a(\lambda)+b(\lambda)\partial)A$ for some polynomials $a(\lambda), b(\lambda)\in \C[\lambda]$,  and $[B_\lambda B]=Q(\lambda, \partial)A+(2\lambda+\partial)\alpha B$ for some skew-symmetric polynomial $Q(\lambda, \partial)$ and for $\alpha\in \{0, 1\}$.  \Cref{conditions} implies that we only need to consider the following three cases:
\begin{description}
	\item[\textbf{Case 1}] $\alpha=0$, $b(\lambda)=0$.
	\item[\textbf{Case 2}] $\alpha=1$, $a(\lambda)=b(\lambda)=0$.
	\item[\textbf{Case 3}] $\alpha=1$, $b(\lambda)=1$ and $a(\lambda)=c\lambda+d$ for some $c, d\in \C$. \\ 
 We divide \textbf{Case 3} into two subcases:
	\textbf{3a: $d\neq 0$} and 
	\textbf{3b: $d=0$}.
\end{description}
Note that for all these three cases, the Jacobi identities for the triples $(A, A, A)$,  $(A, A, B)$ and $(A, B, B)$ are easily verified. Moreover, the coefficients of $B$ in both sides of the Jacobi identity for the triple $(B, B, B)$ are equal. So we only need to consider the coefficients of $A$ in the Jacobi identity for the triple $(B,B,B)$, which gives us the following equation:
\begin{align}\label{Jacobifinalcase}
  \alpha(\lambda-\mu)&Q(\lambda+\mu,\partial)-Q(\lambda,-\lambda-\mu)\left(a(-\lambda-\mu-\partial)+b(-\lambda-\mu-\partial)\partial\right)\nonumber\\
   &=Q(\mu,\lambda+\partial)(a(\lambda)+b(\lambda)\partial)+\alpha(2\mu+\lambda+\partial)Q(\lambda,\partial)\nonumber\\&\quad -Q(\lambda,\mu+\partial)(a(\mu)+b(\mu)\partial)-\alpha(2\lambda+\mu+\partial)Q(\mu,\partial).
 \end{align}

Hence the classification of non-semisimple rank two Lie conformal algebras is equivalent to the classification of the quadruples $(a(\lambda), b(\lambda), \alpha, Q(\lambda, \partial) )$, such that $a(\lambda), b(\lambda), \text{ and }\alpha$ satisfy the conditions in one of the cases described above, and the skew-symmetric polynomial  $Q(\lambda, \partial)$ satisfies (\ref{Jacobifinalcase}).

\subsection{\textbf{Case 1}}\textbf{$(\alpha=0,  b(\lambda)=0)$}
\begin{prop}\label{nilpotent}
	Let $R=\C[\partial]A\oplus \C[\partial]B$ be a rank two Lie conformal algebra satisfying $[A_\lambda A]=0$, $[B_\lambda A]=a(\lambda)A$, and $[B_\lambda B]=Q(\lambda, \partial)A$. If $a(\lambda)=0$, then $Q(\lambda, \partial)$ may be any skew-symmetric polynomial. If $a(\lambda)\neq 0$, then we can find a new basis $\{A, C\}$, such that $[A_\lambda A]=[C_\lambda C]=0$ and $[C_\lambda A]=a(\lambda)A$. 
\end{prop}
\begin{proof}
When $a(\lambda)=0$,  $A$ is a central element in $R$ and $R'\subseteq \C[\partial]A$. So all the Jacobi identities are trivially satisfied and the only constraint for $Q(\lambda, \partial)$ is that it should be skew-symmetric. 
When $a(\lambda)\neq 0$, we show that by a suitable change of basis we can kill $Q(\lambda, \partial)$. If $Q(\lambda, \partial)\neq 0$, we assume that $Q(\lambda, \partial)=\sum_{i\geq 0}^nf_i(\lambda)\partial^i$ and $f_n(\lambda)\neq 0$. Then
(\ref{Jacobifinalcase}) gives us 
	\begin{align}\label{firstcase}
	a(\mu)\sum_{i\geq 0}^n f_i(\lambda)&(\mu+\partial)^i-a(\lambda)\sum_{i\geq 0}^nf_i(\mu)(\lambda+\partial)^i\nonumber\\ &=a(-\lambda-\mu-\partial)\sum_{i\geq 0}^nf_i(\lambda)(-\lambda-\mu)^i .
	\end{align}
Comparing the degrees with respect to $\partial$ on both sides of  (\ref{firstcase}), we have $\deg a(\lambda)\leq \deg_\partial\,Q(\lambda, \partial)$. If $\deg a(\lambda)< n$, then the coefficient of $\partial^n$ on the left side of (\ref{firstcase}) must be zero, which implies that $f_n(\mu)a(\lambda)= f_n(\lambda)a(\mu)$. So $f_n(\lambda)=ka(\lambda)$ for some constant $k\in \C^\times$. For $B'=B-k\partial^nA$, $\{A, B'\}$ forms a new basis and $[B'_\lambda A]=a(\lambda)A$, $[B'_\lambda B']=Q'(\lambda, \partial)A$, where  
$$Q'(\lambda, \partial)=Q(\lambda, \partial) +(-\lambda)^n a(-\lambda-\partial)k-(\lambda+\partial)^na(\lambda)k. $$
It is clear that $\deg_\partial Q'(\lambda, \partial)\leq n$. Moreover, the coefficient of $\partial^n$ in $Q'(\lambda, \partial)$ is $f_n(\lambda)-ka(\lambda)=0$, i.e., $\deg_\partial Q'(\lambda, \partial)<\deg_\partial Q(\lambda, \partial)$.  By induction,  we can assume that $\deg a(\lambda) = m = \deg_\partial Q(\lambda, \partial)$. Since $Q(\lambda, \partial)$ is skew-symmetric, \Cref{keycorollary} implies that the coefficient of $\partial^m$ in  $Q(\lambda, \partial)$ is a polynomial in $\lambda$ of degree less than or equal to $m-1$. Let us assume that $f_m(\lambda)=\sum_{i=0}^{m-1} b_{i} \lambda^i$. If $p(\partial)=\partial^m+\sum_{i=0}^{m-1}p_i\partial^i$, and $C=B+p(\partial) A$, then $[C_\lambda A]=a(\lambda)A$ and 
	\begin{equation*}
	Q_{C, C}^A(\lambda, \partial)=Q(\lambda, \partial) -p(-\lambda) a(-\lambda-\partial)+p(\lambda+\partial)a(\lambda).
	\end{equation*}
	Note that $\deg_\partial Q_{C,C}^A(\lambda, \partial)\leq m$. If we write $a(\lambda)=\sum_{i= 0}^m a_i \lambda^i$ with $a_m\neq 0$, then the coefficient of $\partial^m$ in the polynomial $Q_{C,C}^A(\lambda, \partial)$ is 
	\begin{equation*}
	\sum_{i\geq 0}^{m-1} (b_{i}+(-1)^{m+i+1}a_m p_i+a_i)\lambda^i.
	\end{equation*} 
	Since $a_m\neq 0$, given $a_i, b_{i}$, we can choose $p_i=(-1)^{m+i}\dfrac{a_i+b_{i}}{a_m}$ such that the coefficient of $\partial^n$ in $Q_{C,C}^A(\lambda, \partial)$ becomes zero, which implies that $\deg_\partial Q_{C,C}^A(\lambda, \partial)<m$. But if $Q_{C,C}^A(\lambda, \partial)$ is non-zero, it must also satisfy (\ref{firstcase}), i.e., $\deg\, a(\lambda)\leq \deg\,Q_{C, C}^A(\lambda, \partial)$. Hence $Q_{C, C}^A(\lambda, \partial)=0$ and in the new basis $\{A, C\}$, we have $[A_\lambda A]=[C_\lambda C]=0, [C_\lambda A]=a(\lambda)A$. 
\end{proof}
\begin{remark}
	 It is clear that  $R$ is solvable but not nilpotent when $[A_\lambda A]=[B_\lambda B]=0$ and $[B_\lambda A]=a(\lambda)A$ for some non-zero  polynomial $a(\lambda)$.  So we denote it by $R_{sol}(a(\lambda))$.  When  $[A_\lambda A]=[B_\lambda A]=0$ and $[B_\lambda B]=Q(\lambda, \partial)A$ for some skew-symmetric polynomial $Q(\lambda, \partial)$,  $R$ is nilpotent.  So we denote it by $R_{nil}(Q(\lambda, \partial))$. We will call the $\lambda$-brackets in  \Cref{nilpotent} the normalized $\lambda$-brackets.
\end{remark}
\begin{remark}
As in Case 2 and Case 3, we have $B\in R^{(n)}$ for any $n$, so \Cref{nilpotent} classifies all solvable rank two Lie conformal algebras.  
\end{remark}
\begin{definition}
Two polynomials $f$ and $g$ are said to be associated if $f=k g$ for some $k\in \C^\times$, which  we denote by $f\sim g$. 
\end{definition}
\begin{lemma}\label{isonilpotent}
The solvable Lie conformal algebras $R_{sol}(a(\lambda))$ and $R_{sol}(a'(\lambda))$ are isomorphic if and only if $a(\lambda)\sim a'(\lambda)$.  The nilpotent  Lie conformal algebras $R_{nil}(Q(\lambda, \partial))$ and $R_{nil}(Q'(\lambda, \partial))$ are isomorphic if and only if $Q(\lambda, \partial)\sim Q'(\lambda, \partial)$. 
\end{lemma}
\begin{proof}
Let $R_1$ denote $R_{sol}(a(\lambda))$ or $R_{nil}(Q(\lambda, \partial))$ and $R_2$  denote $R_{sol}(a'(\lambda))$ or $R_{nil}(Q'(\lambda, \partial))$, respectively, with basis elements $\{A, B\}$ and $\{A', B'\}$ satisfying the normalized $\lambda$-brackets. It is clear that if $a(\lambda)=ka'(\lambda)$, then the map $A\mapsto A', B\mapsto kB'$ gives an isomorphism between $R_{sol}(a(\lambda))$ and $R_{sol}(a'(\lambda))$, and if $Q(\lambda, \partial)=kQ'(\lambda, \partial)$, then the map $A\mapsto A', B\mapsto \sqrt{k}B'$ gives an isomorphism between $R_{nil}(Q(\lambda, \partial))$ and $R_{nil}(Q'(\lambda, \partial))$. \\
For the converse, assuming that $\varphi$ is an isomorphism between $R_1$ and $R_2$, we have $\varphi (R_1')\cong R_2'$. In both cases, we have $$\{0\} \subsetneq R_1'\subseteq \C[\partial]A, \qquad \{0\} \subsetneq R_2'\subseteq \C[\partial]A'.$$ 
So $\varphi(A)=f(\partial)A'$ for some polynomial $f(\partial)$.  But $\{\varphi(A), \varphi(B)\}$ forms a basis of $R_2$, so we have $f(\partial)=s$ and $\varphi(B)=tB+g(\partial)A$ for some constants $s,t\in \C^\times$ and some polynomial $g(\partial)$. When $R_1=R_{sol}(a(\lambda))$ and $R_2=R_{sol}(a'(\lambda))$, as $\varphi([B_\lambda A])=[\varphi(B)_\lambda \varphi(A)]$, we get $a(\lambda)=ta'(\lambda)$. When $R_1=R_{nil}(Q(\lambda, \partial))$ and $R_2=R_{nil}(Q'(\lambda, \partial))$, as $\varphi([B_\lambda B])=[\varphi(B)_\lambda \varphi(B)]$, we get $Q(\lambda, \partial)=\dfrac{t^2}{s}Q'(\lambda, \partial)$.
\end{proof}
\subsection{\textbf{Case 2}}\textbf{$(\alpha=1, a(\lambda)=b(\lambda)=0)$}
\begin{prop}\label{semidirect}
Let $R=\C[\partial]A\oplus \C[\partial]B$ be a rank two Lie conformal algebra satisfying $[A_\lambda A]=[B_\lambda A]=0$, and $[B_\lambda B]=(2\lambda+\partial)B+Q(\lambda, \partial)A$. Then by a suitable change of basis we can set $Q(\lambda, \partial)=0$.
\end{prop}
\begin{proof}
If $Q(\lambda, \partial)\neq 0$, then by skew-symmetry, we can assume that $Q(\lambda,\partial)=(2\lambda+\partial)\sum_{i\geq 0}^nf_i(\lambda)\partial^i$ with $f_n(\lambda)\neq 0$. From (\ref{Jacobifinalcase}), we get
\begin{align}\label{JacobiBBB}
    (2\mu+\lambda+\partial)&(2\lambda+\partial)\sum_{i\geq 0}^nf_i(\lambda)\partial^i-(2\lambda+\mu+\partial)(2\mu+\partial)\sum_{i\geq 0}^nf_i(\mu)\partial^i\nonumber\\
    &=(\lambda-\mu)(2\lambda+2\mu+\partial)\sum_{i\geq 0}^nf_i(\lambda+\mu)\partial^i.
\end{align}
Setting $\mu=0$ in (\ref{JacobiBBB}), we have
\begin{equation*}
\sum_{i\geq 0}^nf_i(\lambda)\partial^i=\sum_{i\geq 0}^nf_i(0)\partial^i,
\end{equation*}
i.e.,  $f_i(\lambda)=f_i(0)$ for all $i$. If $B'=B+\sum_{i\geq 0}^nf_i(0)\partial^iA$, then $[B'_\lambda B']=(2\lambda+\partial)B'$ and $[A_\lambda A]=[B'_\lambda A]=0$.
\end{proof}
\begin{lemma}\label{centralelement}
	Let $R=\C[\partial]A\oplus \C[\partial]B$ be a rank two Lie conformal algebra satisfying $[A_\lambda A]=0, [B_\lambda A]=(c\lambda+d+\partial)A$, and $[B_\lambda B]=(2\lambda+\partial)B+Q(\lambda, \partial)A$. Then $R$ has no non-trivial central element.  
\end{lemma}
\begin{proof}
Let $X=f(\partial)A+g(\partial)B$ be central in $R$. As $[X_\lambda A]=g(-\lambda)[B_\lambda A]=0$ and $[X_\lambda B]=f(-\lambda)[A_\lambda B]=0$, we have that $f(\partial)=g(\partial)=0$. Hence $X=0$.  
\end{proof}
\begin{remark}\label{cvandrcdqnonisomorphic}
	It is clear that there exists a non-trivial central element in the Lie conformal algebras considered in Case 2. Hence the Lie conformal algebras considered in Case 2 and those in Case 3 are not isomorphic.	
\end{remark}

\subsection{\textbf{Case 3a}}\textbf{($\alpha=1$, $b(\lambda)=1$, $a(\lambda)=c\lambda+d$ for  $c\in \C, d\in \C^\times$)}

 Assuming that $Q(\lambda,\partial)=\sum_{i=0}^m f_i(\lambda)\partial^i$, from (\ref{Jacobifinalcase}) in Case 3 we have
\begin{align}\label{jacobicequal0}
(c\lambda+c\mu+&(c-1)\partial-d)\sum_{i=0}^{m} f_i(\lambda)(-\lambda-\mu)^i+(\lambda-\mu)\sum_{i=0}^{m} f_i(\lambda+\mu)\partial^i\nonumber
\\&= (c\lambda+d+\partial)\sum_{i=0}^{m} f_i(\mu)(\lambda+\partial)^i+(2\mu+\lambda+\partial)\sum_{i=0}^{m} f_i(\lambda)\partial^i\nonumber\\ 
&\quad -(c\mu+d+\partial)\sum_{i=0}^{m} f_i(\lambda)(\mu+\partial)^i -(2\lambda+\mu+\partial)\sum_{i=0}^{m} f_i(\mu)\partial^i.
\end{align}
By skew-symmetry, it is clear that if $\deg_\partial Q(\lambda,\partial)\leq 1$, then $Q(\lambda,\partial)=(2\lambda+\partial)k$ for some constant $k\in \C$. It is easy to see that such a polynomial always satisfies (\ref{jacobicequal0}). 
Hence we may assume that  $\deg_\partial Q(\lambda,\partial)=m\geq 2$.  Comparing the coefficients of $\partial^m$ in both sides of (\ref{jacobicequal0}), we get
\begin{equation}\label{leadingcoefficientequals0}
(\lambda-\mu)f_m(\lambda+\mu)=((c+m-2)\lambda-\mu+d)f_m(\mu)-((c+m-2)\mu-\lambda+d)f_m(\lambda).
\end{equation} 
When $\deg_\partial Q(\lambda,\partial)=m\geq 3$, comparing the coefficients of $\partial^{m-1}$ in both sides of (\ref{jacobicequal0}), we get
\begin{align}\label{secondleadingcoefficientequals0}
(\lambda-\mu)f_{m-1}(\lambda+\mu)=&\left(cm+\binom{m}{2}\right)\lambda^2  f_m(\mu)+((c+m-3)\lambda-\mu))f_{m-1}(\mu) \nonumber \\
&-\left(cm+\binom{m}{2}\right)\mu^2f_m(\lambda)-((c+m-3)\mu-\lambda))f_{m-1}(\lambda)\nonumber \\
&+d(f_m(\mu)m\lambda-f_m(\lambda)m\mu+f_{m-1}(\mu)-f_{m-1}(\lambda)).
\end{align}
\begin{prop}\label{dnotequal0}
Let $R=\C[\partial]A\oplus \C[\partial]B$ be a rank two Lie conformal algebra satisfying $[A_\lambda A]=0, [B_\lambda A]=(c\lambda+d+\partial )A$ and $[B_\lambda B]=Q(\lambda, \partial)A+(2\lambda+\partial)B$. If $d\neq 0$, then by a suitable change of basis we can set $Q(\lambda, \partial)=0$. 
\end{prop}
\begin{proof}
Firstly, if $\deg_\partial Q(\lambda, \partial)\geq 2$, then we show that by a suitable change of basis we may decrease its degree with respect to $\partial$. Let $Q(\lambda,\partial)=\sum_{i\geq 0}^mf_i(\lambda)\partial^i$ with $f_m(\lambda)\neq 0$ and $m\geq 2$. 
Setting $\lambda=0$ in (\ref{leadingcoefficientequals0}), we get 
\begin{equation*}\label{keyequationford=0}
f_m(\mu)d-((c+m-2)\mu+d)f_m(0)=0,
\end{equation*}
i.e., $f_m(\mu)=\dfrac{f_m(0)}{d}((c+m-2)\mu+d)$ is of degree less or equal to 1.

For $B'=B-k\partial^{m}A$, we have $[B'_\lambda A]=[B_\lambda A]$  and
\begin{align*}
[B'_\lambda B']
&=(2\lambda+\partial)B'+Q_{B',B'}^A(\lambda,\partial)A
\end{align*}
where
\begin{align*}
 Q_{B',B'}^A(\lambda,\partial)=&(2\lambda+\partial)k\partial^m+Q(\lambda,\partial)\nonumber\\
 &-(k(-\lambda)^m(c\lambda+c\partial-d-\partial)-k(\lambda+\partial)^m(c\lambda+d+\partial)).
\end{align*}
It is straightforward to see that $\deg_\partial Q_{B',B'}^A(\lambda,\partial)\leq m$. Rewriting $Q_{B',B'}^A(\lambda,\partial)=\sum_{i\geq 0}^m g_i(\lambda)\partial^i$, we get 
 \begin{equation*}
 g_m(\lambda)=f_m(\lambda)-k[(c+m-2)\lambda +d] .
 \end{equation*} 
Choosing $k=\dfrac{f_m(0)}{d}$, we conclude that $g_m(\lambda)=0$, i.e., $\deg_\partial Q_{B', B'}^A(\lambda, \partial)< \deg_\partial Q(\lambda, \partial)$. This procedure can be continued until we have $\deg_\partial Q(\lambda, \partial)=1$. By induction, we may assume that $ Q(\lambda, \partial)=(2\lambda+\partial)k$ for some constant $k$. Setting $C=B-\dfrac{k}{d}\partial A$, we get $[C_\lambda A]=(c\lambda+d+\partial)A$ and $[C_\lambda C]=(2\lambda+\partial)C$.
\end{proof}
\subsection{\textbf{Case 3b. }}\textbf{($\alpha=1$, $b(\lambda)=1$, $a(\lambda)=c\lambda$ for $c\in \C$)}

 In this subsection, we assume that $R=\C[\partial]A\oplus \C[\partial]B$ satisfying $[A_\lambda A]=0$, $[B_\lambda A]=(c\lambda +\partial)A$ and $[B_\lambda B]=Q(\lambda, \partial)A+(2\lambda+\partial)B$. We write $Q(\lambda, \partial)=\sum_{n \geq 0} Q_n(\lambda, \partial)$ where $Q_n(\lambda, \partial)$ is the homogeneous component of $Q(\lambda, \partial)$ of degree $n$ as a polynomial in $\lambda $ and $\partial$. Note that  (\ref{Jacobifinalcase}) is homogeneous in Case 3b in the sense that $Q(\lambda, \partial)$ satisfies (\ref{Jacobifinalcase}) if and only if $Q_n(\lambda, \partial)$ does so for all $n$. In the same sense,  skew-symmetry is also a homogeneous property. Thus $Q(\lambda, \partial)$ satisfies (\ref{Jacobifinalcase}) and the skew-symmetry if and only if $Q_n(\lambda, \partial)$ does so for all $n$.

We show that after suitable changes of basis, we can kill almost all $Q_n(\lambda, \partial)$ except a few special values of $n$ depending on the parameter $c$. We also determine  $Q_n(\lambda, \partial)$ explicitly for those special values which we discuss in two cases, $n\leq 4$ and $n\geq 5$. Note that $Q_0(\lambda, \partial)\equiv 0$ by skew-symmetry. When $n\leq 4,$ by skew-symmetry, the polynomials $Q_n(\lambda, \partial)$ must be of the following forms (see \Cref{Skew-symmetry}),
\begin{align}
Q_1(\lambda, \partial)&=\alpha_1(2\lambda+\partial) ,\label{eq:degree1} \\ 
Q_2(\lambda, \partial)& =\alpha_2(2\lambda+\partial)\partial, \label{eq:degree2}\\
Q_3(\lambda, \partial)&=(2\lambda+\partial)(\alpha_3 \partial^2 +\beta_3 (\lambda^2+\lambda\partial)),\label{eq:degree3}\\ Q_4(\lambda,\partial)&=(2\lambda+\partial)(\alpha_4 \partial^3 +\beta_4 (\lambda^2+\lambda\partial)\partial),\label{eq:degree4}
\end{align}
where $\alpha_1, \alpha_2, \alpha_3, \beta_3, \alpha_4, \beta_4 \in \C$ are constants.

For $B'=B+k \partial^{m-1} A$, where  $k\in \C$ and $m\geq 1$, $\{A, B'\}$ forms a new basis of $R$ satisfying $[A_\lambda A]=0, [B'_\lambda A]=(c\lambda+\partial)A,$ and $[B'_\lambda B']=(2\lambda+\partial)B'+Q'(\lambda, \partial)A$ such that 
\begin{align}\label{changeofbasisformula}
Q'(\lambda, \partial)&-Q(\lambda, \partial)\nonumber\\
 &=k[(-\lambda)^{m-1}(c(\lambda+\partial)-\partial)+(\lambda+\partial)^{m-1}(c\lambda+\partial)-(2\lambda+\partial)\partial^{m-1}]\nonumber\\
&=k(c+m-3)\lambda\partial^{m-1}+k(m-1)(c+(m-2)/2)\lambda^2\partial^{m-2}+\cdots,
\end{align}
where ``$\cdots$'' stands for some homogeneous polynomial in $\lambda, \partial$ of degree $m$ involving only powers of $\partial$ which are strictly less than $m-2$.
Note that $Q'(\lambda, \partial)$ differs from $Q(\lambda, \partial)$  only in the  homogeneous component of degree $m$. 
\begin{lemma}\label{degreelessorequal4}
	Let $R=\C[\partial]A \oplus \C[\partial]B$ be a rank two Lie conformal algebra satisfying $[A_\lambda A]=0$, $[B_\lambda A]=(c\lambda +\partial)A$ and $[B_\lambda B]=Q(\lambda, \partial)A+(2\lambda+\partial)B$. If we write $Q(\lambda, \partial)=\sum_{n\geq 1} Q_n(\lambda, \partial)$ where $Q_n(\lambda, \partial)$ is the homogeneous component of degree $n$ of $Q(\lambda,\partial)$, then
	\begin{itemize}
		\item[(1)] $Q_1(\lambda, \partial)=\alpha_1(2\lambda+\partial)$ for some $\alpha_1\in \C$. Moreover, if $c \neq 1$, then by a suitable change of basis we can set $\alpha_1= 0$. 
		\item[(2)] $Q_2(\lambda, \partial)=\delta_{c, 0}\alpha_2(2\lambda+\partial)\partial$ for some $\alpha_2\in \C$.
		\item[(3)] $Q_3(\lambda, \partial)=(2\lambda+\partial)(\delta_{c, -1}\alpha_3 \partial^2 +\beta_3 (\lambda^2+\lambda\partial))$ for some $\alpha_3, \beta_3\in \C$. Moreover, if $c\neq 0$, by a suitable change of basis we can set $\beta_3=0$. 
		\item[(4)] $Q_4(\lambda,\partial)=\beta_4(2\lambda+\partial) (\lambda^2+\lambda\partial)\partial$. 
		Moreover, if $c\neq -1$, by a suitable change of basis we can set $\beta_4=0$. 
		\end{itemize}
\end{lemma}
\begin{proof}
The general forms of $Q_n(\lambda, \partial)$ are given in (\ref{eq:degree1})--(\ref{eq:degree4}). 
The coefficients of $\alpha_2, \alpha_3$, i.e., $\delta_{c, 0}, \delta_{c, -1}$,  and $\alpha_4\equiv 0$ are by straightforward calculations.

For $B'=B+(k_0+k_1\partial+k_2\partial^2+k_3\partial^3)A$, $\{A, B'\}$ forms a new basis of $R$ with $[A_\lambda A]=0, [B'_\lambda A]=(c\lambda+\partial)A$ and $[B'_\lambda B']=(2\lambda+\partial)B'+Q'(\lambda, \partial)A$ such that   
		\begin{equation*}
		Q'(\lambda, \partial) = Q(\lambda, \partial) +(2\lambda+\partial)[k_0(c-1)+k_2c(\lambda^2+\lambda\partial)+k_3(c+1)(\lambda^2+\lambda\partial)\partial].
		\end{equation*}
If we write $Q'(\lambda, \partial)=\sum_{n\geq 1}Q_n'(\lambda, \partial)$, where $Q_n'(\lambda, \partial)$ is the homogeneous component of degree $n$ of $Q'(\lambda,\partial)$, then $Q_n'(\lambda, \partial)=Q_n(\lambda, \partial)$ for $n\geq 5$ and $n=2$. 

If $c\neq 1$, we may choose $k_0$ suitably to get $Q'_1(\lambda, \partial)=0$. If $c\neq 0$, we may choose $k_2$ suitably to set $\beta_3=0$ in the general form of $Q'_3(\lambda, \partial)$ given in (\ref{eq:degree3}). If $c\neq -1$, $k_3$ may be suitably chosen to set $\beta_4=0$ in the general form of $Q'_4(\lambda, \partial)$ given in (\ref{eq:degree4}).
\end{proof}

Next we consider the case $n\geq 5$. If we write $Q_n(\lambda, \partial)=\sum_{i=0}^n b_i\lambda^{n-i}\partial^i$. Then   
(\ref{Jacobifinalcase}) for $Q_n(\lambda, \partial)$ gives us 
\begin{align}\label{jacobicequal0homogeneous}
(\lambda-\mu)&\sum_{i=0}^{n} b_i(\lambda+\mu)^{n-i}\partial^{i}+(c\lambda+c\mu+(c-1)\partial)\sum_{i=0}^{n} b_i\lambda^{n-i}(-\lambda-\mu)^i\nonumber
\\&= (c\lambda+\partial)\sum_{i=0}^{n} b_i\mu^{n-i}(\lambda+\partial)^i+(2\mu+\lambda+\partial)\sum_{i=0}^{n} b_i\lambda^{n-i}\partial^i\nonumber\\ 
&\quad -(c\mu+\partial)\sum_{i=0}^{n} b_i\lambda^{n-i}(\mu+\partial)^i -(2\lambda+\mu+\partial)\sum_{i=0}^{n} b_i\mu^{n-i}\partial^i.
\end{align} 
\begin{lemma}\label{constanttermzero}
If $n\geq 4$, then $Q_n(0, \partial)=0$, i.e., $\deg_\partial Q_n(\lambda, \partial)\leq n-1$. Moreover, if $n$ is even, we have $Q_n(\lambda, 0)=0$.
\end{lemma}
\begin{proof}
Let $Q_n(\lambda, \partial)=\sum_{i=0}^nb_{i}\lambda^{n-i} \partial ^{i}$. Setting $\lambda=0$ in (\ref{jacobicequal0homogeneous}), we get 
\begin{align}\label{eq:finalcase}
(c\mu+(c-1)\partial)b_n(-\mu)^n
= (2\mu+\partial) b_n\partial^n -(c\mu+\partial) b_n(\mu+\partial)^n.
\end{align} 
As $n\geq 4$, comparing the coefficients of $\mu^2\partial^{n-1}$ and $\mu^{n-1}\partial^2$ in both sides of (\ref{eq:finalcase}), we get  $$\left(\binom{n}{2}+cn\right)b_n=\left(c\binom{n}{2}+n\right)b_n=0.$$  But  $\binom{n}{2}+cn$ and $c\binom{n}{2}+n$ can not be zero at the same time when $n\geq 4$. Hence $b_n=0$, i.e.,  $Q_n(0, \partial)=0$. 
Since $Q_n(\lambda, \partial)$ is a skew-symmetric polynomial, by \Cref{Skew-symmetry}, we can write  $$Q_n(\lambda, \partial)=(2\lambda+\partial)\sum_{i\geq 0}k_i (\lambda^2+\lambda\partial)^i \partial^{n-1-2i}$$
where $k_i\in \C$. If $n$ is even, then $n-1-2i>0$. Hence $Q_n(\lambda, 0)=0$ if $n$ is even.
\end{proof}
The following lemma will be very important in the sequel. 
\begin{lemma}\label{secondkeylemma}
If $Q_n(\lambda, \partial)\neq 0$ and $n\geq 4$, then $\deg_\partial Q_n(\lambda, \partial)\geq n-3$. Moreover, $\deg_\partial Q_n(\lambda, \partial)= n-2$ or $n-3$ only if $c+n-3=0$.  
\end{lemma}
\begin{proof}
Let us assume that $\deg_\partial Q_n(\lambda, \partial)=m$ and $Q_n(\lambda, \partial)=\sum_{i=0}^mb_{i}\lambda^{n-i}\partial^i$ such that $b_m\neq 0$. Since $Q_n(\lambda, \partial)$ is skew-symmetric, by \Cref{keycorollary}, the coefficient of $\partial^m$ in $Q_n(\lambda,\partial)$ is of degree less than or equal to $m-1$, i.e.,  $n-m\leq m-1$.  Hence $m\geq \lfloor \frac{n+1}{2}\rfloor$ which implies that $m\geq 2$ since $n\geq 4$. By (\ref{leadingcoefficientequals0}), we have
\begin{equation*}
b_m h_{n-m}^m(c, \lambda, \mu)=0,
\end{equation*}  
where $
h_j^m(c, \lambda, \mu):=(c+m-2)\lambda-\mu) \mu^j-((c+m-2)\mu-\lambda)\lambda^j-(\lambda-\mu)(\lambda+\mu)^j.$
By straightforward calculations, we have $h_1^m(c, \lambda, \mu)\equiv 0$ for all $c$  and $ h_j^m(c, \lambda, \mu)\neq 0$ for all $j\geq 4$. Moreover, we have 
\begin{align}
&h_0^m(c, \lambda, \mu)=0 \quad \text{~if and only if~}  c+m-2=0,\label{secondcondition1}\\
&h_2^m(c, \lambda, \mu)=0 \quad \text{~if and only if~}  c+m-1=0,\label{mequalsnminus2}\\
&h_3^m(c, \lambda, \mu)=0 \label{thirdcondition}\quad \text{~if and only if~}  c+m=0.
\end{align}
As $b_m h_{n-m}^m(c, \lambda, \mu)=0$ and $b_m\neq 0$,  (\ref{secondcondition1})--(\ref{thirdcondition}) imply that $n-m\leq 3$, i.e., $\deg_\partial Q_n(\lambda, \partial)=m \geq n-3$ which proves the first part of the lemma. 

If $m=n-2$, then  $c+m-1=0$, i.e., $c+n-3=0$ by (\ref{mequalsnminus2}). 

If $m=n-3$, then $c+m=0$, i.e., $c+n-3=0$ by (\ref{thirdcondition}). 
\end{proof}
\begin{cor}\label{killdegree4}
Let $n\geq 4$ and $n\neq 3-c$. Then by a suitable change of basis we can kill $Q_n(\lambda, \partial)$.
\end{cor}
\begin{proof}
By \Cref{constanttermzero}, we have $\deg_\partial Q_n(\lambda, \partial)\leq n-1$. 
We show that we can kill off the  $\lambda\partial^{n-1}$ term in $Q_n(\lambda, \partial)$. If $B'=B+k\partial^{n-1}A$, then $[B'_\lambda B']=(2\lambda+\partial)B'+Q'(\lambda, \partial)A$. Write $Q'(\lambda, \partial)=\sum Q_\ell'(\lambda, \partial)$, where $Q_\ell'(\lambda, \partial)$ is the homogeneous component of degree $\ell$ in $Q'(\lambda, \partial)$. By (\ref{changeofbasisformula}), we have  $Q_\ell'(\lambda, \partial) = Q_\ell(\lambda, \partial)$ for $\ell \neq n$, and
\begin{align*}
Q_n'(\lambda, \partial) = Q_n(\lambda, \partial) +k(c+n-3)\lambda\partial^{n-1}+...,
\end{align*}
where ``...'' stands for some polynomial whose degree with respect to $\partial$ is less than $n-1$.  
Since $n\neq 3-c$, we may choose $k$ properly to kill off the  $\lambda\partial^{n-1}$ term in $Q_n(\lambda, \partial)$, i.e., $\deg_\partial Q_n'(\lambda, \partial)\neq n-1$. Then \Cref{secondkeylemma} implies that $Q_n'(\lambda, \partial)=0$ as $c+n\neq 3$.  
\end{proof}

Thus for $n\geq 5$, we only need to consider the special case where $n=3-c$. 
\begin{lemma}\label{nequals3minusc}
If $n=3-c\geq 5$, then $\deg_\partial Q_n(\lambda, \partial) \neq n-1$. Moreover, by a suitable change of basis we can kill off the  $\lambda^2\partial^{n-2}$ term in $Q_n(\lambda, \partial)$. 
\end{lemma}
\begin{proof}
If $\deg_\partial Q_n(\lambda, \partial) = n-1$, then $Q_n(\lambda, \partial)=\sum_{i\geq 0}^{n-1}b_i \lambda^{n-i}\partial^i$ for some $b_{n-1}\neq 0$. As $n-1\geq 3$, by (\ref{secondleadingcoefficientequals0}), we have (where in our case, $m=n-1,$  $f_m(\lambda)=b_{n-1}\lambda$ , $f_{m-1}(\lambda)=b_{n-2}\lambda^2$ and $d=0$)
\begin{align}\label{n=3-c}
\left(cm+\dfrac{m(m-1)}{2}\right)\lambda^2\mu b_{n-1}-\left(cm+\dfrac{m(m-1)}{2}\right)\mu^2\lambda b_{n-1}=0. 
\end{align} 
But (\ref{n=3-c}) is satisfied only when $c=-1$ and $n=4$. Hence we have $\deg_\partial Q_n(\lambda, \partial) \neq n-1$.

For $B'=B+k\partial^{n-1}A$, we have $[B'_\lambda B']=(2\lambda+\partial)B'+Q'(\lambda, \partial)A$ for some polynomial $Q'(\lambda, \partial)$.  If we write $Q'(\lambda, \partial)=\sum_{\ell \geq 0} Q_\ell'(\lambda, \partial)$, where $Q_\ell'(\lambda, \partial)$ is the homogeneous component of degree $\ell$ of $Q'(\lambda, \partial)$, then by (\ref{changeofbasisformula}), we have 
\begin{align*}
Q_n'(\lambda, \partial) = Q_n(\lambda, \partial) +k(n-1)\left(c+\frac{n-2}{2}\right)\lambda^2\partial^{n-2}+...,
\end{align*}
where ``...'' stands for some polynomial whose degree with respect to $\partial$ is less than $n-2$. Hence we may choose $k$ properly to kill off the $\lambda^2\partial^{n-2}$ term in $Q_n(\lambda, \partial)$.
\end{proof}

 By \Cref{nequals3minusc}, if $n=3-c\geq 5$ then we may assume that $\deg_{\partial}Q_{n}(\lambda, \partial)\leq n-3$. Moreover, by \Cref{secondkeylemma} if $Q_{n}(\lambda, \partial)\neq 0$, then $\deg_{\partial}Q_{n}(\lambda, \partial)= n-3$ . 

\begin{remark}\label{uniqueness}
When $n=3-c\geq 5$, if there exists a skew-symmetric and  homogeneous polynomial $Q_n(\lambda, \partial)$ of degree $n$ satisfying  (\ref{Jacobifinalcase}), such that $\deg_\partial Q_n(\lambda, \partial)=n-3$, then it is unique up to scalar multiples. Indeed, if there are two such polynomials, after multiplying by a scalar, we can assume that their leading terms $\lambda^3\partial^{n-3}$ have the same coefficients. Their difference is still skew-symmetric and satisfies   (\ref{Jacobifinalcase}), but is of degree less than or equal to $n-4$  with respect to $\partial$, hence must be zero by \Cref{secondkeylemma}. 
\end{remark}

\begin{cor}\label{cminus2minus3}
If $c=-2, -3$, then by a suitable change of basis we can kill $Q_{3-c}(\lambda, \partial)$. 
\end{cor}
\begin{proof}
By \Cref{nequals3minusc}, if $Q_{3-c}(\lambda, \partial)\neq 0$, we can assume that $\deg_\partial Q_{3-c}(\lambda, \partial)=-c$.  Then the coefficient of $\lambda^{3}\partial^{-c}$ in $Q_{3-c}(\lambda, \partial)$ is non-zero, which contradicts  \Cref{keycorollary} as $3\geq -c$. 
\end{proof}

\begin{lemma}\label{killdegree5ceven}
Let $n=3-c$. Assume that $n \geq 6$ if $n$  is even and $n\geq 11$  if $n$ is odd. If $\deg_\partial Q_n(\lambda, \partial)\leq n-3$, then $Q_n(\lambda, \partial)=0$. 
\end{lemma}
\begin{proof}
Let $Q_n(\lambda, \partial)=\sum_{i=0}^{n-3}b_i\lambda^{n-i}\partial^i$. Comparing the coefficients of $\lambda^2\partial^k\mu^{n-1-k}$ in both sides of (\ref{jacobicequal0homogeneous}),  we have, 
\begin{equation}\label{bkbkplusoneequation}
\left(\binom{n-k}{1}-\binom{n-k}{2}\right) b_k=\left(c\binom{k+1}{1}+\binom{k+1}{2}\right)b_{k+1},
\end{equation}
for $0\leq k\leq n-3$
because the only terms containing the monomial $\lambda^2\partial^k\mu^{n-1-k}$ in (\ref{jacobicequal0homogeneous}) are $(\lambda-\mu)\sum_{i=0}^{n} b_i(\lambda+\mu)^{n-i}\partial^{i}$ and $(c\lambda+\partial)\sum_{i=0}^{n} b_i\mu^{n-i}(\lambda+\partial)^i$ as $b_n=b_{n-1}=b_{n-2}=0$ in our case.

When $n\geq 6$ is even, we have $b_0=0$ by \Cref{constanttermzero}. Note that for $k\leq n-3$, we always have  $c\binom{k+1}{1}+\binom{k+1}{2}\neq 0$. So $b_k=0$ by (\ref{bkbkplusoneequation}), i.e.,  $Q_n(\lambda, \partial)=0$. 

Now let $n=3-c\geq 11$ be odd. 
Comparing the coefficients of $\lambda^4\mu^{n-3}$ on both sides of (\ref{jacobicequal0homogeneous}),  we get
\begin{equation}\label{b0b3bminusc}
b_0\left(\binom{n}{3}-\binom{n}{4}\right)+c(1-c)b_{n-3}=cb_3. 
\end{equation}
Since $c\binom{k+1}{1}+\binom{k+1}{2}\neq 0$ for $k\leq n-3$, (\ref{bkbkplusoneequation}) gives us
\begin{equation*}
b_{k+1}=\dfrac{\Pi_{i= 0}^k(n-i)\left(\dfrac{3+i-n}{2}\right)}{\Pi_{i= 0}^k(i+1)\left(\dfrac{6-2n+i}{2}\right)}b_0=\dfrac{\Pi_{i=0}^k(n-i)(3+i-n)}{\Pi_{i= 0}^k(i+1)(6-2n+i)}b_0.
\end{equation*}
In particular, 
\begin{equation*}
b_{n-3}=\dfrac{\Pi_{i= 0}^{n-4}(n-i)(3+i-n)}{\Pi_{i= 0}^{n-4}(i+1)(6-2n+i)}b_0,
\end{equation*}
and 
\begin{equation*}
b_3=\dfrac{\Pi_{i= 0}^{2}(n-i)(3+i-n)}{\Pi_{i= 0}^{2}(i+1)(6-2n+i)}b_0=\dfrac{n(n-1)(n-2)}{3!}\dfrac{5-n}{4(7-2n)}b_0.
\end{equation*}
Note that $$\dfrac{\Pi_{i= 0}^{n-4}(n-i)}{\Pi_{i= 0}^{n-4}(i+1)}=\dfrac{n(n-1)(n-2)}{3!},\qquad \dfrac{\Pi_{i= 0}^{n-4}(3+i-n)}{\Pi_{i= 0}^{n-4}(6-2n+i)}=\dfrac{(n-3)!\times (n-3)!}{(2n-6)!}, $$ 
and $$\binom{n}{3}-\binom{n}{4}=\dfrac{n(n-1)(n-2)(7-n)}{3!\times 4}.$$
Replacing $b_3$ and $b_{n-3}$ by the above expressions in (\ref{b0b3bminusc}) and then dividing both sides of (\ref{b0b3bminusc}) by $\dfrac{n(n-1)(n-2)}{3!}$, we get 

\begin{equation*}
b_0\left(\dfrac{7-n}{4}+(3-n)(n-2)\dfrac{(n-3)!\times (n-3)!}{(2n-6)!}+\dfrac{(n-3)(5-n)}{4(7-2n)} \right)=0.
\end{equation*}
Note that $\dfrac{7-n}{4}+\dfrac{(n-3)(5-n)}{4(7-2n)}=\dfrac{n^2-13n+34}{4(7-2n)}<0$ when $n\geq 11$. It is obvious that $(3-n)(n-2)\dfrac{(n-3)!\times (n-3)!}{(2n-6)!}<0$ when $n\geq 11$. So we must have $b_0=0$. Then (\ref{bkbkplusoneequation})  implies that $b_k=0$ for all $k$, i.e., $Q_n(\lambda, \partial)=0$. 
\end{proof}

\begin{lemma}\label{n79}
If $c=-4$, then $Q_7(\lambda, \partial)=(2\lambda+\partial)(\lambda^2+\lambda\partial)^3$ satisfies (\ref{jacobicequal0homogeneous}). If $c=-6$, then $Q_9(\lambda, \partial)=(2\lambda+\partial)(11(\lambda^2+\lambda\partial)^4+2(\lambda^2+\lambda\partial)^3\partial^2)$ satisfies (\ref{jacobicequal0homogeneous}).
\end{lemma}
\begin{proof}
This follows from straightforward calculations. Indeed, by the property of skew-symmetry and by the assumption that $\deg_\partial\,Q_{3-c}(\lambda, \partial)=-c$, we can give a general form for $Q_{3-c}(\lambda, \partial)$. Then we just need to check (\ref{jacobicequal0homogeneous}) for $Q_{3-c}(\lambda, \partial)$. 
\end{proof}

\subsection{\textbf{The classification}}

\begin{theorem}\label{maintheorem}
	Let $R$ be a non-semisimple rank two  Lie conformal algebra. Then, up to isomorphism, $R$ is one of the following types.
	\begin{itemize}
		\item[(1)] If $R$ is nilpotent, then 
		$R\cong R_{nil}(Q(\lambda, \partial))$, and  $R_{nil}(Q(\lambda, \partial))\cong R_{nil}(Q'(\lambda, \partial))$ if and only if $Q(\lambda, \partial)=kQ'(\lambda, \partial)$ for some $k\in \C^\times.$
		
		\item[(2)] If $R$ is solvable but not nilpotent, then $R\cong R_{sol}(a(\lambda))$, and  $R_{sol}(a(\lambda))\cong R_{sol}(a'(\lambda))$ if and only if $a(\lambda)=ka'(\lambda)$ for some $k\in \C^\times.$
		
		\item[(3)] If $R$ is not solvable, then we have two classes. 
		\begin{itemize}
			\item[(3i)] $R$ is the direct sum of a rank one commutative Lie conformal algebra and the Virasoro Lie conformal algebra. 
			\item[(3ii)]  $R$ has a basis $\{A, B\}$, such that $[A_\lambda A]=0, [B_\lambda A]=(c\lambda+d+\partial)A$ and $[B_\lambda B]=Q_c(\lambda, \partial)A+(2\lambda+\partial)B$ for some constants $c, d\in \C$ and some skew-symmetric polynomial $Q_c(\lambda, \partial)$. We denote such $R$ in this class by $ R(c, d, Q_c(\lambda, \partial))$.  Moreover, $Q_c(\lambda, \partial)\neq 0$ only when $d=0$ and $c\in \{1, 0, -1, -4, -6 \}$, in which case we document the explicit formulae for $Q_c(\lambda,\partial)$ in the following table.

			\begin{tabular}{|c|c|}
				\hline 
				$c$ &  $Q_c(\lambda, \partial)$, $\beta, \gamma \in \C$,  \\ 
				\hline 
				$1$ &  $\beta(2\lambda+\partial)$  \\ 
				\hline 
				$0$ &  $\beta(2\lambda+\partial)(\lambda^2+\lambda\partial)+\gamma(2\lambda+\partial)\partial$  \\ 
				\hline 
				$-1$ &  $\beta(2\lambda+\partial)\partial^2+\gamma(2\lambda+\partial)(\lambda^2+\lambda\partial)\partial$  \\ 
				\hline 
				$-4$ &  $\beta(2\lambda+\partial)(\lambda^2+\lambda\partial)^3$  \\ 
				\hline
				$-6$ &  $\beta(2\lambda+\partial)[11(\lambda^2+\lambda\partial)^4+2(\lambda^2+\lambda\partial)^3\partial^2]$  \\ 
				\hline 
			\end{tabular}

			\noindent  Moreover, $R(c, d, Q_c(\lambda, \partial))\cong R(c', d', Q_{c'}'(\lambda, \partial))$ if and only if $c=c', d=d'$ and $Q_c(\lambda, \partial)=kQ_{c'}'(\lambda, \partial)$ for some $k\in \C^\times$. 
		\end{itemize} 
	\end{itemize}
\end{theorem}
\begin{proof}
	It is clear that Lie conformal algebras of different types in $(1)$-$(3)$ are non-isomorphic.  The non-semisimple rank two Lie conformal algebras are divided in three cases by \Cref{conditions}.  The solvable case (Case 1) is done in \Cref{nilpotent} and  \Cref{isonilpotent}, which gives us the types $(1)$-$(2)$ in the list. 	For the non-solvable ones, we have Case 2 and Case 3. Case 2 is done in  \Cref{semidirect} and it gives us the class $(3i)$. We have divided Case 3 into two subcases. Case 3a is done in \Cref{dnotequal0} and it gives  us the $d=0$ and $Q_c(\lambda, \partial)=0$ part of the class $(3ii)$. 
	
	The Case 3b is the core of our work. Recall that in this subcase, we assume that $R$ has a basis $\{A, B\}$ satisfying $[A_\lambda A]=0, [B_\lambda A]=(c\lambda+\partial )A$ and $[B_\lambda B]=Q(\lambda, \partial)A+(2\lambda+\partial)B$. Let us write $Q(\lambda,\partial)=\sum_{n\geq 1} Q_n(\lambda, \partial)$, where $Q_n(\lambda, \partial)$ is the homogeneous component of degree $n$ of $Q(\lambda, \partial)$ as a polynomial in $\lambda$ and $\partial$.   
	
   If $c\notin \{1, 0, -1, -4, -6\}$, then by \Cref{degreelessorequal4}, we can kill $Q_n(\lambda, \partial)$ for $n\leq 4$. By  \Cref{killdegree4},  \Cref{cminus2minus3} and \Cref{killdegree5ceven},  we can kill $Q_n(\lambda, \partial)$ for $n\geq 5$. Thus by a suitable change of basis, we can set $Q(\lambda, \partial)=0$. 
    
    If $c\in \{1, 0, -1\}$, then by \Cref{killdegree4} we can kill $Q_n(\lambda, \partial)$ for $n\geq 5$. For $n\leq 4$, \Cref{degreelessorequal4} determines the parts of $Q_n(\lambda, \partial)$ which can be killed.

  If $c\in \{-4, -6\}$, then by  \Cref{degreelessorequal4}, we can kill $Q_n(\lambda, \partial)$ for $n\leq 4$. By \Cref{killdegree4}, we can kill $Q_n(\lambda, \partial)$ for $n\geq 5$ and $n\neq 3-c$. For $n=3-c$, by \Cref{nequals3minusc}, we can assume that $\deg_\partial Q_{3-c}(\lambda, \partial)=-c$ if it is not zero, and is unique up to scalar multiples by \Cref{uniqueness}. Then by \Cref{n79}, we get the formulae for $Q_{3-c}(\lambda, \partial)$ as listed in the table.

We prove the isomorphism between different types of Lie conformal algebras in the class $(3ii)$ in \Cref{isononsolvable}.
\end{proof}

\begin{remark}
	The dimension of the solution space of the polynomials $Q_c(\lambda, \partial)$ in the table of \Cref{maintheorem} were predicted by Theorem 7.2  in \cite{BakKacVoronov}, where the cohomology of the Virasoro Lie conformal algebra with coefficients in rank one modules was calculated.
	
\end{remark}
\begin{remark}
We call the $\lambda$-brackets in \Cref{maintheorem} the normalized $\lambda$-brackets.
\end{remark}

\begin{lemma}\label{finalcor}
	If $Q_c(\lambda, \partial)\neq 0$ is one of the polynomials appearing in the table of  \Cref{maintheorem}, then  for any $f(\partial)\in \C[\partial]$, \begin{equation}\label{eq:lemma2.23}
	Q_c(\lambda, \partial)\neq  f(\lambda+\partial)(c\lambda+\partial)+f(-\lambda)(c\lambda+c\partial-\partial)-(2\lambda+\partial)f(\partial).
	\end{equation}
\end{lemma}
\begin{proof}
It is enough to prove the lemma for the homogeneous components of $Q_c(\lambda, \partial)$ as a polynomial in $\lambda$ and $\partial$.  Denote the right side of (\ref{eq:lemma2.23}) by  $S^c_{f(\partial)}(\lambda)$ for a polynomial $f(\partial)$.  Let $h(\lambda, \partial)$ be a homogeneous component of $Q_c(\lambda, \partial)$ of degree $m+1$.  If $h(\lambda, \partial)=S^c_{f(\partial)}(\lambda)$ for some polynomial $f(\partial)$, then $f(\partial)$ must have a degree $m$ monomial $k\partial^m$ and $h(\lambda, \partial)=S^c_{k\partial^m}(\lambda)$. 

For $c\in \{0, 1, -1\}$, the total degree of $Q_c(\lambda, \partial)$ is less than or equal to $4$. Hence if $Q_c(\lambda, \partial)=S^c_{f(\partial)}(\lambda)$ for some polynomial $f(\partial)$, then we can assume that the degree of $\deg\, f(\partial)$ is less or equal to 3. For $f(\partial)=\sum_{i=0}^3k_i\partial^3$, we have 
$$S^c_{f(\partial)}(\lambda)=(2\lambda+\partial)[k_0(c-1)+k_2c(\lambda^2+\lambda\partial)+k_3(c+1)(\lambda^2+\lambda\partial)\partial].$$
Now it is clear that $Q_c(\lambda, \partial)\neq S^c_{f(\partial)}(\lambda)$ for any $f(\partial)$.

For $c\in \{-4, -6\}$ and $Q_c(\lambda, \partial)\neq 0$, the total degree of $ Q_c(\lambda, \partial)$ is $3-c$ and $\deg_\partial Q_c(\lambda, \partial)=-c$. Thus we can assume  $f(\partial)=k\partial^{2-c}$ if $Q_c(\lambda, \partial)=S^c_{f(\partial)}(\lambda)$. But by straightforward calculations, we have $\deg_\partial S^c_{k\partial^{2-c}}(\lambda)=1-c$ if $k\neq 0$.  

\end{proof}
\begin{prop}\label{isononsolvable}
The rank two Lie conformal algebras $R(c, d, Q_c(\lambda, \partial))$ and $R(c', d', Q_{c'}'(\lambda, \partial))$ are isomorphic if and only if $c=c', d=d'$ and $Q_c(\lambda, \partial)=kQ_{c'}'(\lambda, \partial)$ for some $k\in \C^\times.$ 
\end{prop}	 
\begin{proof}
Let $\{A, B\}$ and $\{A', B'\}$ be bases of $R(c, d, Q_c(\lambda, \partial))$ and $R(c', d', Q_{c'}'(\lambda, \partial))$, respectively, satisfying the normalized $\lambda$-brackets as in  \Cref{maintheorem}. Assume that $\varphi$ is an isomorphism between $R(c, d, Q_c(\lambda, \partial))$ and $R(c', d', Q_{c'}'(\lambda, \partial))$, such that  $\varphi(A)=f(\partial)A'+g(\partial)B'$ and $\varphi(B)=p(\partial)A'+q(\partial)B'$ for some polynomials $f(\partial),g(\partial),p(\partial) \text{ and } q(\partial) \in \C[\partial]$. Then the equation  $$\varphi([A_\lambda A])=0=[\varphi(A)_\lambda \varphi(A)]$$ implies that $g(\partial)=0$. Since $\{\varphi(A), \varphi(B)\}$ forms a basis, we have $f(\partial)=s$ and $q(\partial)=t$ for some $s,t\in \C^\times $.  Moreover, $\varphi([B_\lambda A])=[\varphi(B)_\lambda \varphi(A)]$ implies that $t=1$ and $c=c', d=d'$.

When $d=d'\neq 0$, there is nothing to say, since  $Q_c(\lambda, \partial)=Q_{c'}'(\lambda, \partial)=0$. 

When $d=d'=0$, from $\varphi([B_\lambda B])=[\varphi(B)_\lambda \varphi(B)]$, we get $$Q_c(\lambda, \partial)-sQ_{c'}'(\lambda, \partial)=p(\lambda+\partial)(c\lambda+\partial)+p(-\lambda)(c\lambda+c\partial-\partial)-(2\lambda+\partial)p(\partial).$$ 
Since $c=c', d=d'$, $Q_c(\lambda, \partial)-sQ_{c'}'(\lambda, \partial)$ is also a polynomial of the form as listed in the table of  \Cref{maintheorem}. Thus by \Cref{finalcor}, we have $Q_c(\lambda, \partial)=sQ_{c'}'(\lambda, \partial)$. 

The converse is clear by defining the map $\varphi(A)= k^{-1} A'$ and $\varphi(B)= B'$. 
\end{proof}

\section{Automorphism groups of rank two Lie conformal algebras}\label{sec:3}

We devote this final section to compute the automorphism groups of rank two Lie conformal algebras.  We denote by $R_c$,  $R_{ss}$ and $R_{cs}$ the commutative, the semisimple and the direct sum of the  commutative and the Virasoro rank two Lie conformal algebras, respectively.  It is clear  that    $\Aut\, R_c\cong GL(2, \C[\partial])$.

Let $\{L^1, L^2\}$ be a basis of  $R_{ss}$ satisfying $[L^i_\lambda L^j]=\delta_{i, j}(2\lambda+\partial)L^i$. Assume that $\varphi(L^i)=f_i(\partial)L^1+g_i(\partial)L^2$ is an automorphism of $R_{ss}$, where $f_i(\partial), g_i(\partial)\in \C[\partial]$. 
Then we can show that for $i, j\in \{1, 2\}$, we have
\begin{equation}\label{fgij}
\delta_{i, j}f_i(\partial)=f_i(-\lambda)f_j(\lambda+\partial) \quad\mbox{~and~} \quad
\delta_{i, j}g_i(\partial)=g_i(-\lambda)g_j(\lambda+\partial).
\end{equation}
 Setting $i=j=1$ and  $i=j=2$ in (\ref{fgij}), we get $f_i(\partial), g_i(\partial)\in \{0, 1\}$ for $i\in \{1, 2\}$. Setting $i=1, j=2$ in (\ref{fgij}), we get $f_1(-\lambda)f_2(\lambda+\partial)=g_1(-\lambda)g_2(\lambda+\partial)=0$. So $$\begin{pmatrix}
	f_1(\partial)& g_1(\partial)\\
	f_2(\partial)& g_3(\partial)
\end{pmatrix}=\begin{pmatrix}
1&0\\
0& 1
\end{pmatrix}\quad \mbox{~or~} \quad \begin{pmatrix}
0&1\\
1& 0
\end{pmatrix},$$ i.e., $\Aut\, R_{ss}\cong \Z_2$. Let $\{K, L\}$ be a basis of $R_{cs}$ satisfying $[K_\lambda K]=[K_\lambda L]=0$ and $[L_\lambda L]=(2\lambda+\partial)L$. Let $\phi$ be an automorphism of $R_{cs}$. Then direction calculations show that $\phi(L)=L$ and $\phi(K)=cK$ for some $c\in \C^*$, i.e., $\Aut\,R_{cs}\cong \C^\times.$

\begin{lemma}\label{keylemmaforautomorphisms}
	Let $R$ be a non-commutative rank two Lie conformal algebra of type $R_{nil}(Q(\lambda, \partial)), R_{sol}(a(\lambda))$ or $R(c, d, Q_c(\lambda, \partial))$ with basis $\{A, B\}$ satisfying the normalized $\lambda$-brackets in \Cref{maintheorem}.  Let $\varphi\in \Aut\, R$ be an automorphism. Then $\varphi(A)=k_1A$  and $\varphi(B)=k_2B+p(\partial)A$ for some $k_1, k_2\in \C^\times$ and some  $p(\partial)\in \C[\partial]$. 
\end{lemma}

\begin{proof}
	Assume that $\varphi(A)=f(\partial)A+g(\partial)B$ and $\varphi(B)=p(\partial)A+q(\partial)B$, where $f(\partial), g(\partial), p(\partial), q(\partial)\in \C[\partial]$. We show that $g(\partial)=0$. Then $\{\varphi(A), \varphi(B)\}$ forms a basis of $R$ will imply that $f(\partial)$ and  $q(\partial)$ are some nonzero constants. 
	
	For $R=R(c, d, Q_c(\lambda, \partial))$, we have $g(\partial)=0$ because  $$ 0=[\varphi(A)_\lambda \varphi(A)]\equiv g(-\lambda)g(\lambda+\partial)(2\lambda+\partial) B {\mod \C[\partial]A}.$$
	 
	For $R=R_{nil}(Q(\lambda, \partial))$ and $Q(\lambda, \partial)\neq 0$, we have $g(\partial)=0$ because  $$0=[\varphi(A)_\lambda \varphi(A)]=g(-\lambda)g(\lambda+\partial)Q(\lambda, \partial)A.$$ 
	
	For $R=R_{sol}(a(\lambda))$ and $a(\lambda)\neq 0$, we have $g(\partial)=0$ because  $$a(\lambda)\varphi(A)=[\varphi(B)_\lambda \varphi(A)]\equiv 0 {\mod \C[\partial]A}.$$
	\end{proof}
 
\begin{lemma}\label{finallemma}
If $c, d\in \C$ and  $f(\partial)\in \C[\partial]$ satisfy
\begin{equation}\label{changeofbasisequation}
f(\lambda+\partial)(c\lambda+d+\partial)+f(-\lambda)(c\lambda+c\partial-d-\partial)-(2\lambda+\partial)f(\partial)=0,
\end{equation} then there exists some $a_0, a_1, a_2, a_3, k\in \C$, such that 

$$f(\partial)=\begin{cases}
k\left(1+\dfrac{1-c}{d}\partial\right) & \mbox{~if~} d\neq 0,\\
\delta_{c, 1}a_0+a_1\partial+\delta_{c, 0}a_2\partial^2+\delta_{c, -1}a_3\partial^3 &\mbox{~if~} d=0.
\end{cases}$$
 
\end{lemma}
\begin{proof}
For $d\neq 0$, set $\lambda=0$ in (\ref{changeofbasisequation}), we get  $f(\partial)d+f(0)((c-1)\partial-d)=0,$ i.e., $f(\partial)=k\left(1+\dfrac{1-c}{d}\partial\right)$, where $f(0)=k\in \C$. By straightforward calculations, we can see  that such a polynomial always satisfies (\ref{changeofbasisequation}).

For $d=0$, let us assume that $ f(\partial)=\sum_i a_i \partial^i$. Note that  (\ref{changeofbasisequation}) is equivalent to 
\begin{equation*}
a_i\left[(\lambda+\partial)^i(c\lambda+\partial)+(-\lambda)^i(c\lambda+c\partial-\partial)-(2\lambda+\partial)\partial^i\right]=0
\end{equation*} 
for all $i$. By straightforward calculations, we have $$(\lambda+\partial)^i(c\lambda+\partial)+(-\lambda)^i(c\lambda+c\partial-\partial)-(2\lambda+\partial)\partial^i\neq 0\mbox{~for~} i\geq 4.$$ Thus $a_i=0$ for $i\geq 4$, and $f(\partial)=a_0+a_1\partial+a_2\partial^2+ a_3\partial^3$. Then  (\ref{changeofbasisequation}) gives us
$$a_0(c-1)(2\lambda+\partial)+a_2c(2\lambda+\partial)(\lambda^2+\lambda\partial)+a_3(c+1)(2\lambda+\partial)(\lambda^2+\lambda\partial)\partial=0, $$
i.e.,  $a_1\in \C$ and  $a_0(c-1)=a_2c=a_3(c+1)=0$.   
\end{proof}
 \begin{theorem}\label{maintheoremautogroup}
 The automorphism groups of rank two Lie conformal algebras are:
 \begin{itemize}
 	\item[(1)] $\Aut \,R_c\cong GL(2, \C[\partial]), \Aut\, R_{ss}\cong \Z_2$, and $\Aut\,R_{cs}\cong  \C^\times$. 
 
 	\item[(2)]  $\Aut\, R_{nil}(Q(\lambda, \partial))\cong \C^\times \ltimes \C[\partial]$. 
 	
  	\item[(3)] $\Aut \, R_{sol}(a(\lambda))\cong \C^\times\ltimes \C$. 
  	\item[(4)]  
  	$$\Aut\, R(c, d, Q_c(\lambda, \partial))\cong \begin{cases}
  	\C^\times\ltimes \C &\mbox{~if~} d\neq 0, \\
  	\C &\mbox{~if~} d=0, c\notin \{1, 0, -1\}, Q_c(\lambda, \partial)\neq 0,\\
  	\C^\times\ltimes \C   &\mbox{~if~} d=0, c\notin \{1, 0, -1\}, Q_c(\lambda, \partial)=0,\\
  	\C^2  &\mbox{~if~} d=0, c\in \{1, 0, -1\}, Q_c(\lambda, \partial)\neq 0,\\
  	\C^\times \ltimes \C^2   &\mbox{~if~} d=0, c\in \{1, 0, -1\}, Q_c(\lambda, \partial)= 0.  
  	\end{cases}$$ 
 \end{itemize}
 \end{theorem}
\begin{proof}
Part $(1)$ is done in the beginning of this section. For $(2)$--$(4)$, we denote by $\{A, B\}$ the basis satisfying the normalized $\lambda$-brackets in \Cref{maintheorem}. For example,  we assume that $[A_\lambda A]=0, [B_\lambda A]=(c\lambda+d+\partial)A$ and $[B_\lambda B]=Q_c(\lambda, \partial)A+(2\lambda+\partial)B$ for $R(c, d, Q_c(\lambda, \partial))$. 

Let $\varphi$ be an automorphism of $R$. Then \Cref{keylemmaforautomorphisms} implies that $\varphi(A)=k_1A$ and $\varphi(B)=k_2B+f(\partial)A$ for some $k_1, k_2\in \C^\times$ and $f(\partial)\in \C[\partial]$. We use the matrix $\begin{pmatrix}
k_1& f(\partial)\\
0& k_2
\end{pmatrix}$ to represent $\varphi$. Note that $\varphi([A_\lambda A])=[\varphi(A)_\lambda \varphi(A)]$ is already satisfied when $\varphi$ is of the above form. So we only need to consider $\varphi([B_\lambda A])=[\varphi(B)_\lambda \varphi(A)]$ and $\varphi([B_\lambda B])=[\varphi(B)_\lambda \varphi(B)]$. 

For $R=R_{nil}(Q(\lambda, \partial))$, $\varphi([B_\lambda A])=[\varphi(B)_\lambda \varphi(A)]$ is satisfied and $\varphi([B_\lambda B])=[\varphi(B)_\lambda \varphi(B)]$ implies that $k_1=k_2^2$. Thus $$\Aut\, R_{nil}(Q(\lambda, \partial))\cong \left\{\begin{pmatrix}
k_2^2& f(\partial)\\
0& k_2
\end{pmatrix}~|~ k_2\in \C^\times,  f(\partial)\in \C[\partial]\right\}\cong \C^\times \ltimes \C[\partial]. $$

For $R=R_{sol}(a(\lambda))$,  $\varphi([B_\lambda A])=[\varphi(B)_\lambda \varphi(A)]$ implies that $k_2=1$ and then $\varphi([B_\lambda B])=[\varphi(B)_\lambda \varphi(B)]=0$ implies that $f(\lambda+\partial)a(\lambda)=f(-\lambda)a(-\lambda-\partial)$. Since $a(\lambda)\neq 0$, we get $f(\partial)=ka(-\partial)$ for some $k\in \C$. 
Thus
$$\Aut\, R_{sol}(a(\lambda))\cong \left\{\begin{pmatrix}
k_1& ka(-\partial)\\
0& 1
\end{pmatrix}~|~ k_1\in \C^\times, k\in \C\right\}\cong \C^\times \ltimes \C. $$

For $R=R(c, d, Q_c(\lambda, \partial))$, $\varphi([B_\lambda B])=[\varphi(B)_\lambda \varphi(B)]$ implies that $k_2=1$ and \begin{equation}\label{finalequation}
(k_1-1)Q_c(\lambda+\partial)=f(\lambda+\partial)(c\lambda+d+\partial)+f(-\lambda)(c\lambda+c\partial-d-\partial)-(2\lambda+\partial)f(\partial).
\end{equation}
Note that when $k_2=1$, $\varphi([B_\lambda A])=[\varphi(B)_\lambda \varphi(A)]$ is also satisfied.

We show that $(k_1-1)Q_c(\lambda+\partial)\equiv 0$. For $d\neq 0$, we already have $Q_c(\lambda, \partial)=0$, so we are done. For $d=0$, if $(k_1-1)Q_c(\lambda, \partial)\neq 0$,  then $Q_c(\lambda, \partial)$ is of the form $S^c_{f(\partial)}(\lambda)$ as in \Cref{finalcor}, which is impossible. Hence both sides of (\ref{finalequation}) must be zero. Thus $k_1=1$ if $Q_c(\lambda, \partial)\neq 0$ and $k_1\in \C^\times$ if $Q_c(\lambda, \partial)=0$. The polynomial  $f(\partial)$ is given in  \Cref{finallemma}. 

For $d\neq 0$,  we have $k_1\in \C^\times$ since $Q_c(\lambda, \partial)= 0$. By \Cref{finallemma}, we have $f(\partial)=k\left(1-\dfrac{c-1}{d}\partial\right)$ for some $k\in \C$, hence
$$\Aut\, R(c, d, 0)\cong \left\{\begin{pmatrix}
k_1& k\left(1-\dfrac{c-1}{d}\partial\right)\\
0& 1
\end{pmatrix} ~|~ k_1\in \C^\times, k\in \C\right\}\cong \C^\times\ltimes \C.$$
Using similar arguments, we have the following list of automorphism groups of $R(c, d, Q_c(\lambda, \partial))$ for $d=0$. 

\begin{itemize}
\item[(i)] If  $c\notin \{1, 0, -1\} \text{ and } Q_c(\lambda, \partial)\neq 0$, then $k_1=1$ and $f(\partial)=k\partial$, hence 
$$\Aut\, R(c, 0, Q_c(\lambda,\partial))\cong \left\{\begin{pmatrix}
1& k\partial\\
0& 1
\end{pmatrix} ~|~ k\in \C \right\}\cong \C. $$
	
\item[(ii)] If $c\notin \{1, 0, -1\} \text{ and } Q_c(\lambda, \partial)=0$, then $k_1\in \C^\times$ and $f(\partial)=k\partial$, hence 
$$\Aut\, R(c, 0, Q_c(\lambda,\partial))\cong \left\{\begin{pmatrix}
k_1& k\partial\\
0& 1
\end{pmatrix} ~|~k_1\in \C^\times, k\in \C \right\}\cong \C^\times \ltimes \C. $$

\item[(iii)] If $c\in \{1, 0, -1\} \text{ and } Q_c(\lambda, \partial)\neq 0$, then $k_1=1$ and $f(\partial)=a_1\partial+\delta_{c, 1}a_0+\delta_{c, 0}a_2\partial^2+\delta_{c, -1}a_3\partial^3$ where $a_i\in \C$, hence $$\Aut\, R(c, 0, Q_c(\lambda,\partial))\cong \left\{\begin{pmatrix}
1& f(\partial)\\
0& 1
\end{pmatrix} \right\}\cong \C^2. $$

\item[(iv)] If $c\in\{1, 0, -1\} \text{ and } Q_c(\lambda, \partial)=0$, then $k_1\in \C^\times$ and $f(\partial)=a_1\partial+\delta_{c, 1}a_0+\delta_{c, 0}a_2\partial^2+\delta_{c, -1}a_3\partial^3$ where $a_i\in \C$, hence 
$$\Aut\, R(c, 0, Q_c(\lambda,\partial))\cong \left\{\begin{pmatrix}
k_1& f(\partial)\\
0& 1
\end{pmatrix} ~|~k_1\in \C^\times \right\}\cong \C^\times \ltimes \C^2. $$
\end{itemize}
\end{proof}

\bibliographystyle{amsplain}

\end{document}